\documentclass{amsart}
\usepackage{bbm}

\usepackage[latin1]{inputenc}

\usepackage{verbatim}

\newtheorem{prop}{Proposition}
\newtheorem{theorem}{Theorem}
\newtheorem{lemma}{Lemma}
\newtheorem{defi}{Definition}
\newtheorem{corollary}{Corollary}

\def\N{{\mathbb N}}
\def\R{{\mathbb R}}
\def\Z{{\mathbb Z}}
\def\P{{\mathbb P}}
\def\E{{\mathbb E}}
\def\T{{\mathbb T}}
\def\S{{\mathbb S}}

\def\cal{\mathcal}
\def\eps{\varepsilon}
\def\etal{{\em et al.}}

\newcommand\ind[1]{\mathbbm{1}_{\{#1\}}}
\newcommand\croc[1]{\left\langle #1\right\rangle}

\title[The Evolution of a Spatial Stochastic Network]{The Evolution of a Spatial Stochastic Network}

\author[Ph. Robert]{Philippe  Robert}
\address[Ph. Robert]{INRIA Paris --- Rocquencourt,  Domaine de Voluceau, 78153 Le Chesnay, France.}
\email{Philippe.Robert@inria.fr}
\urladdr{http://www-rocq.inria.fr/\string~robert}

\date{02/08/2010}


\begin{document}
\begin{abstract}
The asymptotic  behavior of a stochastic network represented  by a birth  and death
processes of particles on a compact state space is analyzed.  Births:  Particles are created at
rate $\lambda_+$  and their location is  independent of the  current configuration. Deaths
are due to negative particles arriving at rate $\lambda_-$. The death of a particle occurs
when  a negative  particle arrives  in  its neighborhood  and kills  it.  Several  killing
schemes are  considered.  The  arriving locations of  positive and negative  particles are
assumed to have  the same distribution. By using a  combination of monotonicity properties
and invariance  relations it  is shown  that the configurations  of particles  converge in
distribution for several models.  The problems  of uniqueness of invariant measures and of
the  existence   of  accumulation  points   for  the  limiting  configurations   are  also
investigated.  It is  shown for several natural models  that if $\lambda_+<\lambda_-$ then
the asymptotic configuration has a finite  number of points with probability $1$. Examples
with  $\lambda_+<\lambda_-$ and  an infinite  number of  particles in  the limit  are also
presented.
\end{abstract}

\maketitle

\bigskip

\hrule

\vspace{-3mm}

\tableofcontents

\vspace{-1cm}

\hrule

\bigskip

\section{Introduction}\label{secintro}
In this  paper one  studies the asymptotic  behavior of  configurations of points  in a
compact state space $H$. Two types of particles, ``$+$'' and ``$-$'', 
arrive on $H$ according to some arrival process.   
\begin{itemize}
\item A "$+$" particle stays at  its arriving site $x$, adding therefore a new point at
  the location $x$  to the current configuration. 
\item A ``$-$'' particle arriving at $x$ kills a  point of the configuration if there is
  one  in a  specified  neighborhood of $x$. In any case  a ``$-$'' particle disappears.  
\end{itemize}
These natural models of birth and death processes of particles occur in several domains.
\begin{enumerate}
\item {\em Queueing Systems.}\\
When there are only a finite number of possible locations $x_1$, \ldots, $x_p$ for the
particles, this model is equivalent to $p$ single server queues where $+$ are jobs
arriving in one of the queues of the system and  the arrival rate of $-$'s at location
$x_k$, $1\leq k\leq p$ is simply the service rate of the $k$th queue. 
\item  {\em Stochastic networks.}\\ 
In the context of a wireless network, the ``$+$'' particles represent the requests for
transmission at some location. A ``$-$'' particle is the capacity of service available in the
neighborhood of a point at some moment. The assumption that the closest $+$ of a $-$ is
transmitted is an approximation of the fact that, for energy dissipation reasons, this $+$
transmits with the highest rate. See Foss~\cite{Foss} for related mathematical models 
and Serfozo~\cite{Serfozo}.
\item {\em Biological Networks}.\\ Growth models of protein networks can be represented by
points in a three-dimensional cube, the birth of points corresponding to the local
extension of filaments.  See Beil \etal~\cite{Beil} and L\"uck \etal~\cite{Lueck}.
\item {\em Matching problems in theoretical computer science.}\\ For multi-dimensional on-line bin
  packing  problems. In this setting, a $+$  is an
item $A$ which is alone in its bin, once an item whose size matches ``well'' with that of $A$, both
of them are stored in the bin and $A$ is therefore removed. See Coffman
\etal\cite{Coffman}, Robert~\cite{Robert:03} and  Karp \etal~\cite{Karp} for example.   
\item {\em Statistics.}\\ Ferrari \etal~\cite{Ferrari} presents a simulation method of the 
  invariant distribution of reversible processes on configurations of points in $\R^d$ with constant
birth rate  and whose death rate of a particle is proportional to its multiplicity.
\end{enumerate}
In the following $H$ is a compact metric space and ${\cal M}(H)$ denotes the space of
non-negative Radon measures on $H$ carried by points. See Rudin~\cite{Rudin:01}  and
Dawson~\cite{Dawson:16} for the main definitions and results on ${\cal M}(H)$.
When the state $M_0$ of the initial configuration has $n$ points $z_1$, $z_2$, \ldots,
$z_n$, it will be described as an element of ${\cal M}(H)$, a point process on $H$,
i.e. $M_0=\delta_{z_1}+\delta_{z_2}+\cdots+\delta_{z_n}$, where $\delta_x$ is the Dirac
mass at $x\in H$. 

Starting from $M_0$, if the next particle arrives at the location $X_1$ and its class is given by
$I_1\in\{+,-\}$, then the next state of the configuration is $M_1$ defined by
\begin{equation}\label{ev1}
M_1=M_0+\ind{I_1=+}\delta_{X_1}-\ind{I_1=-}\delta_{t_1(X_1,M_0)},
\end{equation}
where $t_1(X_1,M_0)$ is the (possible) location of the point of $M_0$ which is removed. It
may happen that no point is removed, the Dirac mass at $t_1(X_1,M_0)$ is understood as the $0$ measure
on $H$ in this case. Several definitions for $t_1(X_1,M_0)$ are now presented depending on
the model considered. The function $(x,y)\mapsto d(x,y)$ denotes the distance of the metric space $H$.

\subsubsection*{The Cases of Local Interaction} 
In this setting a ``$-$'' particle arriving at $x\in H$ can only  kill a point located in  a ball
$B(x,1)$ of radius $1$ around $x$. If $B(x,1)$ does not contain a point of the
configuration, the ``$-$'' disappears.  
\medskip
\begin{enumerate}
\item {\em Local Greedy Policy (LG)}. The point $t_1(x,M)$ is the point of $M$ which is
  the closest to $x$ and at distance less than $1$: i.e.  a $y\in H$ such that
\[
d(x,y)=\inf\{d(x,z): z\in H, M(\{z\})\not=0 \text{ and } d(x,z)<1\}.
\]
By convention, if there is not such a point the Dirac measure $\delta_{t_1(x,M)}$ is
defined as the $0$ measure.  In the case there are several points achieving the above
infimum, $t_1(x,M)$ is chosen at random among the corresponding locations. 
\smallskip
\item {\em Local Random Policy  (LR)}. The point $t_1(x,M)$ is chosen at random in the
  subset $\{z\in H: M(\{z\})\not=0 \text{ and } d(x,z)<1\}$.
\smallskip
\item {\em Local One-Sided Policy  (LO)}. It is assumed that  $H$ is a subset of $\R^d$
for some $d\geq 1$. The point $t_1(x,M)$ is $y\in H$ such that
\[
d(x,y)=\inf\{d(x,z): z\in H, M(\{z\})\not=0, d(x,z)<1 \text{ and } z\geq x\},
\]
where the inequality $z\geq x$ is understood coordinate by coordinate. These policies
occur in the context of matching problems. See Karp \etal\cite{Karp}.
\end{enumerate}

\subsubsection*{Global Interaction} In this case there is no constraint of locality to kill a
  point. 
\begin{enumerate}
\item {\em Global Greedy Policy (GG)}. The point $t_1(x,M)$ is $y\in H$ such that
\[
d(x,y)=\inf\{d(x,z): z\in H, M(\{z\})\not=0\}.
\]
\item {\em Global One-Sided Policy  (GO)}. The point $t_1(x,M)$ is $y\in H\subset \R^d$ such that
\[
d(x,y)=\inf\{d(x,z): z\in H, M(\{z\})\not=0 \text{ and } z\geq x\},
\]
\end{enumerate}

\subsection*{Related Spatial Processes}
When  the arrival processes  of positive  and negative  particles are  independent Poisson
processes,  the system can  also be  described as  a continuous  time Markov  process. The
associated infinitesimal generator $\Omega$ can  be expressed as follows: For a convenient
functional $F$ on the space ${\cal M}(H)$,
\begin{multline*}
\Omega(F)=\lambda_+\int_H (F(\eta+\delta_x)-F(\eta))\mu(dx)
\\+\lambda_-\int_H (F(\eta-\delta_y)-F(\eta))\delta(x,y,\eta)\eta(dy)\mu(dx),
\end{multline*}
where $\mu$ is the distribution of $X_1$.
For example, for the $LG$ policy, the death rate $\delta$ is defined as
\[
\delta(x,y,\eta)=\ind{d(x,y)<1\wedge d(x,\eta-\delta_y)},
\]
where, $a\wedge b=\min(a,b)$ and, for $y\in H$ and $\mu\in{\cal M}(H)$
\[
d(y,\mu)\stackrel{\text{def.}}{=} d(y,\{z\in H: \mu(\{z\})\not=0\}).
\]

This is  the point of view of  Garcia and Kurtz~\cite{Garcia} where  Markov processes with
general birth rates and constant  death rates are introduced for the evolution of point
processes on a {\em non-compact} state space. See Penrose~\cite{Penrose} for a survey. 
The main problem analyzed in this case  is the construction of a Markov process with
values in the space  of point processes: due to the interaction  and the non-compact state
space, it is not  possible to order the jump instants so that  the existence result is not
straightforward. Additionaly  Penrose~\cite{Penrose} presents  some limit results  (Law of
large numbers and  central limit theorem) but  from a spatial point of  view: the
asymptotic behavior {\em at some fixed time} of some additive  functional of the point
process in a ball whose diameter goes to infinity.

In this  paper, one does not use explicitly the characterization of these processes by
their infinitesimal generator. The  state space  being compact, the existence results are
straightforward. Limit  results are investigated not with  respect to a  spatial component
but with  respect to 
{\em long time behavior}: when does the  configuration converges in law ?  Since the state
space  is not  of finite  dimension,  the classical  tools using  a Markovian approach,
like Lyapunov functions, see Chapter~8 and ~9 of Robert~\cite{Robert:08} for example, seem
to be more difficult to use. The dynamic which is investigated in this paper is specific
but as it  will be  seen it already leads to some non  trivial  problems: in  some cases,
at equilibrium, the associated stochastic process will not live in the space of Radon
measures for example.   See Section~\ref{secmisc}. 

\subsection*{Stability Property}
The arrival rates of "$+$" and "$-$" particles are denoted respectively by $\lambda_+$ and
$\lambda_-$ and the particles are represented  by a sequence $(I_n,X_n)$ where, for $n\geq
1$, $I_n\in\{+,-\}$ is the type of the $n$th particle and $X_n\in H$ is its location. When
there are  more "$-$" particles than  "$+$" particles, i.e.\  $\lambda_+<\lambda_-$, it is
likely  that the  distribution of  the  configuration should  converge to  a random  point
process having  a finite  number of points  with probability  $1$.  This property  will be
referred to as the {\em stability property} of the system.

For the  GG policy,  this property  holds: with some  independence assumptions,  the total
number of points  evolves as a reflected  random walk on integers with  the negative drift
$\lambda_+-\lambda_-$.  In this case, the geometry plays a minor role in the dynamics. 

As noted by Anantharam and Foss, see Foss~\cite{Foss}, the situation is quite different in
the case of local interaction. The stability property is quite challenging to prove, even
for the one-dimensional circle $\T_1(T)$ of length $T>0$ for example.  Mathematically, the
stability property is formally defined as follows: There exists a probability distribution $Q$
on  the  set ${\cal M}(H)$ of finite  Radon measures on $H$  such that
\begin{enumerate}
\item[a)] {\em Invariance}: if $M_0\stackrel{\text{dist.}}{=} Q$, then
  $M_1\stackrel{\text{dist.}}{=} Q$. 
\item[b)]  {\em Convergence}:  if $M_0\in  {\cal  M}(H)$ and  $(M_n)$ is  the sequence  of
  consecutive configurations, then $(M_n)$ converges in distribution to $Q$.
\end{enumerate}
It  is  quite natural  to  consider  a Markovian  approach  to  investigate the  stability
property.  The sequence $(M_n)$ can be seen  as a Markov chain on ${\cal M}(H)$.  In fact,
as it  will be seen,  the space ${\cal  M}(H)$ will prove to  be too small  to investigate
properly these questions,  a larger space of measures  has to be defined.  One  may try to
prove the Harris  ergodicity property of $(M_n)$ which would  give directly the properties
a) and b).  See Nummelin~\cite{Nummelin:02} for example.  In our case, due to the dynamics
of the process  and the complexity of  the space ${\cal M}(H)$, a  Markovian approach does
not  seem to  work for  a general  state space  $H$.  Furthermore,  as it  will  be shown,
symmetry properties play an important role in  these questions. It does not seem that they
can be really taken into account with a Markovian approach to tackle the general case. 
By using an interesting but specific Lyapunov function, Leskel\"a and Unger~\cite{Leskela}
proposed recently  an alternative proof of  the stability property in the case of the
one-dimensional torus. 

For the existence of an invariant distribution, the approach used in this paper
consists in replacing the problem of finding a distribution $Q$ which is invariant by
Equation~\eqref{ev1} by the problem of existence of a random variable $M$ on ${\cal M}(H)$
such that  
\begin{equation}\label{statev1}
M\circ\theta=M+\ind{I_1=+}\delta_{X_1}-\ind{I_1=-}\delta_{t_1(X_1,M_0)},
\end{equation}
holds  almost surely,  where $\theta$  is  a shift  operator on  a convenient  probability
space. This method goes back to  Loynes~\cite{Loynes:01} to study the stability of a
reflected random walk associated to a stationary sequence.

Loynes's method (1962), which  can be seen as a backward coupling,  has been used to study
stochastic   recursions  in  $\R^d_+$   associated  to   several  queueing   systems.  See
Neveu~\cite{Neveu:14} (1983)  and the references  therein.  Robert~\cite{Robert:03} (1987)
used it  to study a  bin-packing algorithm  related to  the GO policy  defined above.  Propp and
Wilson (1996) used this method (under the  name ``coupling from the past'') in the context
of the Ising model.  See Levin  \etal~\cite{Levin} for a discussion on backward couplings.
As always with backward couplings, a monotonicity property is the main ingredient to prove
the existence of  a random measure $M$ solution of  Equation~\eqref{statev1}. It turns out
that  the existence result  holds in  quite general  framework. 
See  Borovkov  and  Foss~\cite{Foss2}  for  a  general presentation  of  the  analysis  of
stochastic recursions.

 In  a second  step, {\em   invariance relations  and symmetry properties} provide  key
 arguments to  prove the main results  of  the  paper, i.e.\  that  such  a  $M$  is
 unique and  that  the  convergence property~(b) holds. These invariance relations have an
 important impact as it will be seen since the  solution of  Equation~\eqref{statev1} have
 a  finite mass  with probability~$1$ when  they hold.  On  the  other hand,  one
 exhibits examples  for  which these  symmetry relations fail and the solution $M$ has an
 infinite mass with probability~$1$. 

\subsection*{Outline of the Paper}
The paper is mainly devoted to  policies with local interactions when the proportion $p_+$
of "$+$", $p_+=\lambda_+/(\lambda_++\lambda_-)$  is less than $1/2$.  Section~\ref{secnot}
introduces the main definitions and notations  required, in particular, to deal with point
processes   which  may   have  an   infinite   number  of   points.  Theorem~\ref{ex}   of
Section~\ref{equisec}  shows that  under general  conditions  that there  always exists  a
random  point  process  with  possible   accumulation  points  such  that  the  invariance
relation~a) holds  for all  policies listed  above.  In this  general setting,  a stronger
result is  shown for the  local random  policy: the invariant  point process has  a finite
number   of  points   with  probability   $1$,   i.e.\  the   stability  property   holds.
Section~\ref{HG} considers an homogeneous case when $H$ is a  compact metric group,
like the  $d$-dimensional torus  or the  $d$-dimensional sphere. It  is proved  that there
exists a unique  random point process $M$ satisfying  Relation~\eqref{statev1} which has a
finite mass with probability $1$ and such that convergence property~b) holds.  This proves
in particular the stability property for homogeneous state spaces for all policies with
local  interaction.   Section~\ref{sectorus}  studies  a  simple  non-homogeneous  setting
$H=[0,T]$ for the LG  policy. It is proved that the stability  property also holds in this
case.        Section~\ref{secmisc}       considers       one-sided       policies       on
$H=[0,T_1]\times\cdots\times[0,T_d]$, it  is shown that  Properties a) and  b) also hold  in the
case  but  with  a limiting  point  process  having  an  infinite  number of  points  with
probability $1$.

\subsection*{Acknowledgments}
The author wishes to thank Serguei Foss for a discussion on this subject and for the
reference Foss~\cite{Foss}.

\section{Evolution Equations of Point Processes}\label{secnot}
The main notations and definitions concerning point processes are first introduced. 
See Rudin~\cite{Rudin:01} for the general definitions and results on Radon measures,
Dawson~\cite{Dawson:16} for an introduction to random point processes and
Neveu~\cite{Neveu:12} on stationary point processes.   

\subsection{Point Processes}
It is assumed throughout the paper that $H$  is a compact metric space (think of a bounded
closed  subset of  $\R^d$ for  example). A  point  process $M$  on $H$  is a  non-negative
Borel  measure  on  $H$  carried by points, i.e., such  that,  for  any Borel  subset  $A$
of  $H$  one  has $M(A)\in\N\cup\{+\infty\}$.  Define ${\cal  M}^*(H)$  as the set of  all point processes. If $M\in{\cal M}^*(H)$, 
$S(M)$ denotes the set of its  accumulation points,
\[
S(M)\stackrel{\text{def.}}{=}\{y\in H: \forall \eps>0, M(B(y,\eps))=+\infty\}. 
\]
Note that $S(M)$ is in particular a closed  set.  The space ${\cal M}(H)$ is the subset of
${\cal  M}^*(H)$ of  Radon non-negative  measures with  finite mass, i.e., the set of
elements $M\in {\cal  M}^*(H)$ such that $M(H)<+\infty$.  As it  will be seen, for  some
policies the state space  ${\cal M}(H)$ is not always  appropriate to study  the
asymptotic behavior  of configurations of  points in $H$.

If $f:H\to\R$ is some Borel function,
\[
\croc{f,M}\stackrel{\text{def.}}{=} \int_H f(x)M(dx),
\]
in particular $\croc{\mathbbm{1}_A,M}=M(A)$ if $A$ is a Borel subset of $A$. 
A sequence of point processes $(M_n)$ in ${\cal M}(H)$  will be said to  converge to
$M\in{\cal M}^*(H)$ if the sequence $(\croc{f,M_n}$ converges in distribution to
$\croc{f,M}$, for any continuous function $f$ with compact support in $H-S(M)$. 

\begin{comment}
Finally, if $\mu$ is some probability distribution on $H$, ${\cal M}^*_\mu(H)$
denotes the subset of elements $P$ of ${\cal M}^*(H)$  with a set $S(P)$ of
accumulation points negligible for $\mu$,  $\mu(S(P))=0$.
\end{comment}

The ordering of point processes is defined as follows. 
\begin{defi}
If $M$ and $P\in{\cal M}^*(H)$, one denotes by $M\ll P$ if  the relation $M(A)\leq P(A)$
holds for any Borel subset  $A$ of $H$.
\end{defi}
\noindent
If $M\ll P$,  the elements in the support of $M$ are in the support of $P$. 

\subsection*{Extension of the definition of the functional $\mathbf{t_1(\cdot,\cdot)}$ on $H\times\mathbf{{\cal M}^*(H)}$}\ \\
For $x\in H$, the variable $t_1(x,M)$ has been defined in Section~\ref{secintro} for $M\in
{\cal M}(H)$, i.e. when the point process has only a finite number of points. Since the
space ${\cal M}(H)$  is not closed for 
the topology of weak convergence, one has to
define it when there are accumulation points.  Furthermore it will allow
\begin{enumerate}
\item  To have a limiting evolution equation for the possible limiting points of the
sequence $(M_n)$ of the successive states of the configuration. 
\item  To properly define the problem of uniqueness of the invariant distribution for Equation~\eqref{ev1}.
\end{enumerate}
The  definition  of $t_1(\cdot,M)$  is extended to an arbitrary  element  $M$  of ${\cal M}^*(H)$. 
For that, one denotes by $\partial$ a cemetery state for which $\delta_\partial$ is the null measure.  
The variable $t_1(x,M)$ is defined as $\partial$  when $M(B(x,1))=0$ (no point to kill)
and in any of the following situations.  
\begin{itemize}
\item LG policy.
\begin{enumerate}
\item  there exists an accumulation point $a{\in} S(M)$ such that $d(x,a){<}1$,\\ $M(\{a\}){=}0$
  and $M(B(x,\eps)){=}0$ for all $\eps\leq d(x,a)$.
\item there exists $0<r<1$ such that $M(B(x,\eps))=0$ for all $\eps\leq r$ and
the set $\{y\in H: d(x,y)=r, M(\{y\})\not=0\}$ is infinite. 
\end{enumerate}
\item LR policy, when $M(B(x,1))=+\infty$;
\item LO policy. Similar to LG policy by replacing balls $B(x,\eps)$, $x\in H$, $\eps>0$ by $B(x,\eps)\cap  \{y:y\geq x\}$.
\end{itemize}
This definition gives the following proposition. 
\begin{prop}\label{convt1}
For the LG, LO, LR policies, if $(M_n)$ is a non-decreasing sequence, for the order $\ll$, of point
processes of ${\cal M}(H)$, such that $M_{n+1}(H)\leq  M_n(H)+1$,  if $M\in{\cal
 M}^*(H)$ is its limit, then the convergence in distribution
\[
\lim_{n\to+\infty} \delta_{t_1(x,M_n)}=\delta_{t_1(x,M)}
\]
holds in ${\cal M}(H-S(M))$.
\end{prop}
\begin{proof}
One considers only the LG policy, the arguments are similar for the other policies. 

If $0<M(B(x,\eps))<+\infty$ for some $\eps>0$, then the sequence $(t_1(x,M_n))$ is constant after some finite
rank so that the convergence trivially holds.  Different cases have to be considered. 
\begin{itemize}
\item If there exists $\eps_0>0$ such that $M(B(x,\eps_0))=0$ and
  $M(B(x,\eps))=+\infty$ for any $\eps>\eps_0$. Under this assumption, this implies that any
  accumulation point of the sequence  $(t_1(x,M_n))$  is an accumulation point of $M$  and
  consequently, in the space ${\cal M}(H-S(M))$,
\[
\lim_{n\to+\infty} \delta_{t_1(x,M_n)}=0=\delta_\partial=\delta_{t_1(x,M)}.
\]
\item Similarly, if there exists $\eps_0>0$ such that $M(B(x,\eps))=0$ for any $\eps<\eps_0$ and
  $M(B(x,\eps_0))=+\infty$. This implies that after some finite rank, the sequence
  $(t_1(x,M_n))$ is in the set $\Delta=\{y:d(x,y)=\eps_0\}$. Since the LG policy chooses the
  point at random on $\Delta$,  as $n$ goes to infinity, the distribution of
  $(t_1(x,M_n))$ will be concentrated around the accumulation points of $M$ in $\Delta$, so
  that the desired convergence will hold for the Dirac masses at the corresponding
  points. 
\end{itemize}

\end{proof}
\subsection{Probabilistic Model}
It is assumed that the arrival times of $+$ [resp. $-$] is a stationary marked point process
$(s_n^+,X_n^+)$ [resp. $(s_n^-,X_n^-)$] on $\R$ and that $(X_n^+)$ and $(X_n^-)$ are independent stationary
sequences with the same distribution (the location of points at the arrival does not
depend on the type). The superposition of the two stationary point processes 
$(s_n^+,+,X_n^+)$ and $(s_n^-,-,X_n^-)$ yields a stationary point process $(s_n,I_n,X_n)$
where $I_n\in\{+,-\}$ is the type of the $n$th particle. Note that $I_n$ is independent of
$X_n$. Under the Palm measure $\P$ of this
stationary point process the sequence $(s_{n+1}-s_n,I_n,X_n)$ is stationary, i.e. its
distribution is invariant with respect to the shift $\theta$ of coordinates. See
Neveu~\cite{Neveu:12} or Robert~\cite{Robert:08}. The two sequences $(I_n)$ and $(X_n)$
will be assumed to be independent. One denotes
by $p_+=\P(I_1=+)$ and $\mu$ is the distribution of $X_1$. 
\subsection{Evolution Equations}
The evolution of the configuration
describing  the system  is represented  as a stochastic process  $(N_n)$ with  values in
${\cal M}(H)$. For $n\in\N$, $N_n$ is the state of the configuration after the $n$th arrival. 
It  is  defined as  follows, $N_0\in{\cal M}(H)$ and the following recurrence holds, for $n\geq 1$, 
\begin{equation}\label{evol}
N_{n}=N_{n-1}+\ind{I_{n}=+}\delta_{X_{n}}-\ind{I_n=-,    N_{n-1}(B(X_{n},1))\not=0}\delta_{t_1(X_{n},N_{n-1})},
\end{equation}
with, for $M\in{\cal M}(H)$ and $x\in H$ such that $M(B(x,1))\not=0$, $t_1(x,M)$ is the
(possible) location of the point of $M$ which is removed in $B(x,1)$ when a $-$ particle
arrives at $x$ in the configuration $M$. See Section~\ref{secintro} for its definition for
the LG, LR and LO policies. 

\subsection{Stationary Evolution Equations}
For  convenience,  the  framework of  ergodic  theory  will  be  used, see  Cornfeld \etal~\cite{Cornfeld:01} for
an introduction. It  can be  assumed that  all these  random  variables are
defined  on  a  probability  space  $(\Omega,{\cal F},\P)$  equipped  with  an  {\em
  automorphism}, i.e.\ an invertible
transformation $\theta:  \Omega\mapsto \Omega$ such  that $\theta$ leaves  the probability
$\P$ invariant, i.e.\ $\theta\circ\P=\P$. In this setting, the relation
\[
(s_n-s_{n-1},X_n,I_n, n\in\Z)(\theta(\omega))=(s_{n+1}-s_{n},X_{n+1},I_{n+1}, n\in\Z)(\omega)
\]
holds  for any $\omega\in\Omega$. The map $\theta$ is the shift for these stationary sequences. 
In particular for $n\in\Z$, $Z_n=Z_1\circ\theta^n$ for $Z\in\{X,I\}$, where
$\theta^n$ is the $n$th iterate of the mapping  $\theta$.  One denotes by ${\cal F}_0$ the
$\sigma$-field generated by the random variables $I_1\circ\theta^n$, $X_1\circ\theta^n$,
$n\leq -1$. It is assumed throughout the paper that the dynamical system $(\Omega,{\cal F},\P,
\theta)$ is {\em ergodic}: any event $A\in{\cal F}$ invariant by $\theta$, that is
$\theta(A)=A$, has either probability $0$ or $1$.   

Additionally,  a family  $(U_F, F  \text{ finite  subset of  } H)$  of  independent random
variables on finite sets is assumed to be  defined to handle the case when the point to be
removed has to be chosen at random among several points.  If $F$ is a finite set, $U_F$ is
a uniformly distributed random variable in $F$. The formal formulation is skipped.

In this setting, a fixed point equation for random point processes is introduced,
a solution $N\in{\cal M}^*(H)$ is such that the relation
\begin{equation}\label{statevol2}
N\circ\theta=N+ \ind{I_{1}=+}\delta_{X_{1}}-\ind{I_1=-,  N(B(X_{1},1))\not=0}\delta_{t_1(X_{1},N)}
\end{equation}
holds almost surely. The distribution of such an $N$ provides an
invariant distribution of the Markov chain  $(N_n)$ defined by Equation~\eqref{evol}. 
Equation~\eqref{statevol2} is the analogue, for point processes, 
of the formulation used by Loynes~\cite{Loynes:01} to analyze Lindley's Equation 
\[
W_n=\max(W_{n-1}+Z_{n-1},0),\quad n\geq 1.
\]
It is reduced in this case to the problem of the existence of a finite random variable $W$
satisfying  the relation
\begin{equation}\label{Lind}
W\circ\theta=\max(W+Z_1,0),
\end{equation}
almost  surely. See  Robert~\cite{Robert:08}.   The  representation in  the
framework   of   ergodic   theory,   i.e.\   with   the  shift   $\theta$,   is   due   to
Neveu~\cite{Neveu:14}.     This formulation    goes    back   to    the    nice    paper
Ryll-Nardzewski~\cite{Ryll:01} for general stationary point processes.  

\subsection*{Invariant Distribution of the Continuous Time Process}
An invariant distribution $Q$ on ${\cal M}^*(H)$ of the Markov chain $(N_n)$ defined by
Equation~\eqref{evol} gives the equilibrium at the instants of arrival of particles.
An invariant distribution $\widetilde{Q}$ on ${\cal M}^*(H)$ for the corresponding
continuous time jump process $(N_t)$  on ${\cal M}^*(H)$ can then be  defined by 
\[
\int_{{\cal M}^*(H)} F(M)\,\widetilde{Q}(dM)=(\lambda_++\lambda_-)
\E\left(s_1\int_{{\cal M}^*(H)} F(M) \,Q(dM)\right),
\]
for any non-negative Borel function $F$ on ${\cal M}^*(H)$. See Neveu~\cite{Neveu:12}. 
% LocalWords:  Borelian Borel

\section{Existence of an Equilibrium}\label{equisec}
The following property is essential to have the existence of an equilibrium distribution
for the evolution equations~\eqref{evol}.
\begin{lemma}[Monotonicity] For the policies LG, LR and LO, if $P_0$ and $Q_0\in{\cal M}(H)$ are such that $P_0\ll Q_0$
then   there   exists  a coupling between any two sequences  $(P_n)$   and   $(Q_n)$
satisfying   the   evolution  equation~\eqref{evol} with  the initial conditions $N_0=P_0$
and $N_0=Q_0$ respectively,  such that the relation $P_n\ll Q_n$ holds for all $n\geq 0$. 
\end{lemma}
\begin{proof}
It is enough to prove the lemma for the first step.

If $I_1=+$ then $P_1=P_0+\delta_{X_1}$ and $Q_1=Q_0+\delta_{X_1}$ so that the relation
$P_1\ll Q_1$ holds. 

Otherwise $I_1=-$, since $P_0(B(X_{1},1))\leq Q_0(B(X_{1},1))$, the only interesting
case is when $Q_0(B(X_{1},1))\not=0$.

\begin{itemize}
\item  $P_0(B(X_{1},1))=0$. A point of $Q_0$ not in the support of $P_0$ is suppressed so
  that $P_1=P_0\ll Q_1$. 
\item  $P_0(B(X_{1},1))\not=0$. The three policies are treated separately.
\begin{itemize}
\item LR policy. Let $U$ a uniformly distributed random variable on the set of points of
  $Q_0$ within $B(X_1,1)$. If $P_0(\{U\})\not=0$, the point $U$ is removed both for $P_0$
  and $Q_0$. Otherwise, $P_0(\{U\})=0$, $U$ is removed from $Q_0$ and a random point of
  $P_0$ within $B(X_1,1)$ is removed. In any case, the relation $P_1\ll Q_1$ holds.
\item LG Policy. It the point of $Q_0$ with minimal distance to $X_1$ belongs also to
  $P_0$ then it is removed for both point process. Otherwise another point of $P_0$ is
  removed, hence $P_1\ll Q_1$ holds. Note that if one has to chose at random among points 
at the same (minimal) distance of $X_1$, one proceeds as for the LR policy.
\item LO Policy. Same argument as for the LG policy.
\end{itemize}
\end{itemize}
The lemma is proved.
\end{proof}
\subsection*{A Solution to the Stationary Evolution Equation}
The asymptotic behavior of the sequence $(N_n)$ defined by Equation~~\eqref{evol} is
analyzed in the following. 

Define the sequence $(\overline{N}_n)$ on the probability space $\Omega$ by induction as
follows, $\overline{N}_0\equiv 0$ 
where, with a slight abuse of notation, $0$ stands for the null point process and, for
$n\geq 1$, for $\omega\in \Omega$, by using the fact that $\theta$ is an automorphism of
the probability space, the point process $N_n(\omega)$ is defined  by
\begin{multline*}
\overline{N}_{n}(\omega) =\overline{N}_{n-1}(\theta^{-1}(\omega))+ \ind{I_{1}(\theta^{-1}(\omega))=+1}\delta_{X_{1}(\theta^{-1}(\omega))}\\-\ind{I_1(\theta^{-1}(\omega))=-1,    
\overline{N}_{n-1}(B(X_{1},1))(\theta^{-1}(\omega))\not=0}\delta_{t_1(X_{1},\overline{N}_{n-1})(\theta^{-1}(\omega))}.
\end{multline*}
or in a more compact form,
\begin{equation}\label{rec1}
\overline{N}_{n}\circ \theta=\overline{N}_{n-1}+ \ind{I_{1}=+1}\delta_{X_{1}}-\ind{I_1=-1,    
\overline{N}_{n-1}(B(X_{1},1))\not=0}\delta_{t_1(X_{1},\overline{N}_{n-1})}.
\end{equation}
\begin{lemma}
The sequence $(\overline{N}_{n})$ is ${\cal F}_0$-measurable and
$(\overline{N}_{n}\circ\theta^n)=(N_n)$, where $(N_n)$ is the sequence defined by the
recurrence~\eqref{evol} with $N_0=0$.  In particular, for $n\geq 1$, the point processes
$N_n$ and $\overline{N}_n$ have the same distribution.
\end{lemma}
\begin{proof}
This is done easily by induction. By using the above relation
\begin{multline*}
\overline{N}_{n}=\overline{N}_{n-1}\circ \theta^{-1}+ \ind{I_{1}\circ \theta^{-1}=+1}\delta_{X_{1}\circ \theta^{-1}}\\
-\ind{I_1\circ \theta^{-1}=-1,    \overline{N}_{n-1}(B(X_{1},1))\circ \theta^{-1}\not=0}\delta_{t_1(X_{1},\overline{N}_{n-1})\circ \theta^{-1}},
\end{multline*}
one gets that $\overline{N}_{n}$ is a functional of the random variables 
\[
(I_1\circ \theta^{-k}, X_{1}\circ \theta^{-k}), \; 1\leq k\leq n,
\]
and therefore ${\cal F}_0$-measurable. By using again the above relation and replacing
$(I_1\circ \theta^{n-1}, X_{1}\circ \theta^{n-1})$ by $(I_{n},X_{n})$, this gives
\begin{multline*}
\overline{N}_{n}\circ \theta^{n}=\overline{N}_{n-1}\circ \theta^{n-1}+ \ind{I_{n}=+}\delta_{X_{n}}\\
-\ind{I_{n}=-,    \overline{N}_{n-1}\circ \theta^{n-1}(B(X_{n},1))\not=0}\delta_{t_1(X_{n},\overline{N}_{n-1}\circ \theta^{n-1})}.
\end{multline*}
hence the sequence $(\overline{N}_{n}\circ \theta^{n})$ satisfies the same
recursion~\eqref{evol} with the zero measure as initial state. It has therefore the same
distribution as $(N_n)$ with $N_0=0$.  The lemma is proved.
\end{proof}
\begin{theorem}[Existence of a Unique Minimal Equilibrium]\label{ex}
If $p_+<1/2$, for the LG and LO policies, there exists a unique random variable $N$
such that the relation 
\begin{equation}\label{statevol}
N\circ\theta=N+ \ind{I_{1}=+}\delta_{X_{1}}-\ind{I_1=-,    N(B(X_{1},1))\not=0}\delta_{t_1(X_{1},N)},
\end{equation}
holds almost surely in the space ${\cal M}^*(H)$ and which is minimal for the order $\ll$:
if $M$ is a random point process satisfying Relation~\eqref{statevol}, then $N{\ll}M$
holds almost surely.  Such a random variable  is ${\cal F}_0$-measurable.   
\end{theorem}
Let $S(N)$ be the (possibly empty) set of accumulation points of $N$. It is important to
note that Relation~\eqref{statevol} is valid as an identity in the set ${\cal M}(H-S(N))$
of Radon measures  on $H-S(N)$.  As it will be seen, $S(N)$ is in fact a deterministic set.
\begin{proof}
As before $(\overline{N}_n)$ denotes the sequence defined by Equation~\eqref{rec1}. 
The proof is done for the LG policy. The arguments for the LO policy work much in the same
way by replacing the open balls $B(x,1)$, $x\in H$, by $B(x,1)\cap \{y\geq x\}$.

\medskip
\noindent
{\em Convergence of the sequence $(\overline{N}_n)$.}\\
Since clearly $\overline{N}_0\ll \overline{N}_1$, the above lemma shows that $\overline{N}_1\ll
\overline{N}_2$ and by induction $\overline{N}_p\ll \overline{N}_{p+1}$ for any $p_+\geq
1$. Consequently, there exists a non-negative random measure $N$, such that  for any
Borel subset $A$ of $H$, 
\[
N(A)=\lim_{p\to+\infty}\uparrow \overline{N}_p(A)
\]
holds almost surely.
The random variable $N$  is ${\cal F}_0$-measurable as an almost sure limit of the ${\cal
  F}_0$-measurable sequence $(\overline{N}_n)$. 
Relation~\eqref{rec1} gives that for any $n\geq 1$
\[
\left|\overline{N}_{n}(A)\circ\theta - \overline{N}_{n-1}(A)\right|\leq 1
\]
consequently,  the set $\{N(A)=+\infty\}$ is invariant by $\theta$ and, hence, of
probability $1$ or $0$ by the ergodicity property. This argument is used repeatedly in the
following. 

\medskip
\noindent
{\em $N$ is a solution of Equation~\eqref{statevol}}.\\
One checks the equation when $X_1=x\in H$, this is a direct consequence of
Proposition~\eqref{convt1}.

If $M$ is a point process satisfying Relation~\eqref{statevol} then clearly $0\ll M$, and
therefore $ \overline{N}_1\circ\theta \ll M\circ\theta$ by Equations~\eqref{statevol}
and~\eqref{rec1}.  By induction one gets that, for any $n\geq 1$, $\overline{N}_n\ll M$
and consequently $N\ll M$. The variable $N$ is minimal for the order $\ll$. 

\medskip
\noindent
{\em $N$ is in ${\cal M}^*(H)$ with probability $1$}.\\
It remains to prove that the set $S(N)$ of
accumulation points of $N$ has almost surely an empty interior. The ergodicity property shows  
that $S(N)$ is a {\em   deterministic} subset of $H$. If $S(N)$ does not have an empty
interior, there is some $x\in H$ and $\eps>0$ such that $B(x,\eps)\subset S(N)$. 

Equation~\eqref{rec1} gives the relation
\begin{multline*}
\overline{N}_{n}\circ \theta(B(x,\eps))-\overline{N}_{n-1}(B(x,\eps))\\=
p\P\left(X_1\in B(x,\eps) \right)-(1-p)\P\left(\overline{N}_{n-1}(B(X_1,1))\not=0,t_1(X_1,\overline{N}_{n-1})\in B(x,\eps) \right).
\end{multline*}
By integrating the above relation  (Note that $N_n(H)$ is bounded by $n$), by using the
invariance of $\theta$ with respect to $\P$ and the monotonicity of the sequence
$(\overline{N}_{n-1}(B(x,\eps)))$, one gets the inequality
\[
\P\left(\overline{N}_{n-1}(B(X_1,1))\not=0,t_1(X_1,\overline{N}_{n-1})\in B(x,\eps)\right)
\leq \frac{p}{1-p}\P\left(X_1\in B(x,\eps) \right)
\]
and hence
\begin{multline*}
\P\left(\overline{N}_{n-1}(B(X_1,1))\not=0,X_1\in B(x,\eps), t_1(X_1,\overline{N}_{n-1})\in B(x,\eps)\right)
\\\leq \frac{p}{1-p}\P\left(X_1\in B(x,\eps) \right).
\end{multline*}
By assumption, the non-decreasing sequence of sets 
\[
O_n=\left\{\overline{N}_{n}(B(X_1,1))\not=0,X_1\in B(x,\eps), t_1(X_1,\overline{N}_{n})\in B(x,\eps)\right\}
\]
is converging to $\cup_n O_n=\{X_1\in B(x,\eps)\}$. By letting $n$ go to infinity in the
above inequality, this gives the relation
\[
\P\left(X_1\in B(x,\eps) \right)\leq \frac{p_+}{1-p_+}\P\left(X_1\in B(x,\eps) \right).
\]
Since $\P\left(X_1\in B(x,\eps) \right)$ is non-zero, otherwise one could not have
accumulation points in $B(x,\eps)$, this yields $p_+\geq 1/2$. Contradiction. The set
$S(N)$ has therefore an empty interior. The theorem is proved. 
\end{proof}
The next result shows that a much stronger statement holds for the local random policy:
the corresponding minimal variable $N$ has almost surely a finite mass. 
\begin{theorem}[Stability of Local Random Policy]\label{theoloc}
If $p_+<1/2$, for the local random policy, there exists a unique minimal random variable $N$
satisfying relation~\eqref{statevol}. 
the point process $N$ has a finite mass with probability $1$, $\P(N\in{\cal M}(H))=1$.
\end{theorem}
\begin{proof}
One has first to check that the limit $N$ of the sequence $(\overline{N}_{n})$ is a
solution of Equation~\eqref{statevol}. 

\begin{itemize}
\item  If $0<N(B(x,1))<+\infty$, the sequence $(t_1(x,\overline{N}_{n}))$ is constant after
some finite rank and so Equation~\eqref{statevol} holds. 

\item Otherwise, if there exists some $0<\eps_0<1$ such that if
 $\eps< \eps_0$ then $N(B(x,\eps))=0$  and if $\eps>\eps_0$ then
  $N(B(x,\eps))=+\infty$. Let $S_1$ be the accumulation points of $N$ in $B(x,1)$. 
 Because of the random choice in
$B(x,1)$, the limit points of the sequence $(t_1(x,\overline{N}_{n}))$ are therefore all
  necessarily on $S_{1}\subset S(N)$. The sequence   of  Dirac measures
  $\delta_{t_1(x,\overline{N}_{n})}$ converges to $0$ in the set ${\cal     M}(H-S(N))$.  
\end{itemize}

Assume that the set $S(N)$ of accumulation points of $N$ is not empty. It is known that it
is deterministic,  denote by 
\[
S^*(N)=\{y\in H: d(y,S(N))<1\},
\]
the set of points at distance less than $1$ of $S(N)$. 
Equation~\eqref{rec1} gives the relation
\begin{multline*}
\E\left[\overline{N}_{n+1} (S^*(N))\right]-\E\left[\overline{N}_{n}(S^*(N))\right]\\=
p\P\left(X_1\in S^*(N) \right)-(1-p)\P\left(\overline{N}_{n}(B(X_1,1))\not=0,t_1(X_1,\overline{N}_{n})\in S^*(N) \right).
\end{multline*}
and therefore, by monotonicity,  the inequality
\begin{align*}
\frac{p}{1-p}\P\left[X_1\in S^*(N)\right] &\geq \P\left[\overline{N}_{n}(B(X_1,1))\not=0,t_1(X_1,\overline{N}_{n})\in S^*(N)\right]\\
&\geq \P\left[\overline{N}_{n}(B(X_1,1))\not=0,X_1\in S^*(N), t_1(X_1,\overline{N}_{n})\in S^*(N)\right].
\end{align*}
By definition of $S(N)$, the set $\{y\in H: N(\{y\})\not=0, d(y,S(N))\geq 1\}$ is almost surely
finite. If $x\in S^*(N)$, then almost surely $N(B(x,1))=+\infty$, so that, because of
the random choice among the points of $\overline{N}_{n}(B(x,1))$,
\[
\lim_{n\to+\infty} \P\left[\overline{N}_{n}(B(x,1))\not=0,  t_1(x,\overline{N}_{n})\in S^*(N)\right]=1.
\]
By using this relation in the above inequality, this gives
\[
\frac{p_+}{1-p_+}\P\left[X_1\in S^*(N)\right] \geq \P\left[X_1\in S^*(N)\right],
\]
and consequently $\P(X_1\in S^*(N))=0$. If $a\in S^*(N)$, because of the dynamics of the
process,  there exists some $\eps>0$ such that $\P(X_1\in B(a,\eps))>0$.
Contradiction. The set $S^*(N)$ is therefore empty. The theorem  is proved. 
\end{proof}
Although the existence result is important in its own right, it
is only a first step to study the stability properties of these systems. Uniqueness and
convergence results turn out to be much more challenging in general when studying
stochastic recursions of the type~\eqref{evol}. This is the main subject of the following
sections. 

\section{Homogeneous State  Spaces}\label{HG}
To stress the fact that only simple invariance relations are used, it is assumed in this section
that $H$ is a compact metrizable group. More specifically, for our study, the following
properties of the state space are used. For the group operation, a multiplicative notation
is used. 
\begin{itemize}
\item[{\em i})] If $x$, $y\in H$ and $r>0$ then $yB(x,r)=B(yx,r)$
\item[{\em ii})] There exists a unique Borel measure  $\mu$ on $H$, the Haar probability measure,
invariant by  group operations $\tau_x:   y\mapsto yx$ for $x\in H$. 
\end{itemize}
See Loomis~\cite{Loomis} or Weil~\cite{Weil} for an introduction. Simple examples of such
a situation are:
\begin{enumerate}
\item For $d\geq1$, the $d$-dimensional torus 
\[
\T_d(T)=\prod_{i=1}^d \R/T_i\Z, 
\]
for $T=(T_i)\in\R_+^d$.
\item For $d\geq1$, the $d$-dimensional sphere $\S_d(T)$,
\[
\S_d(T)=\left\{x=(x_i)\in\R^d: x_1^2+\cdots+x_{d+1}^2=T^2\right\}.
\]
\end{enumerate}
In both cases, the normalized Lebesgue measure on $H$ is the Haar probability
distribution. Various compact groups of matrices provide additional examples of such a 
situation. Note that a related setting was used by Mecke~\cite{Mecke} to derive a key
relation  between the distribution and  the Palm measure of a given stationary point
processes. 

In the  proofs, the local  greedy policy is  assumed. It is  not difficult to see that similar
arguments  can be  used  for the  local one-sided  policy.  Recall that  a strong  result,
Theorem~\ref{theoloc}, has  already been proved  for the local random  policy.  Throughout
the section  the distribution of the  locations of points  is assumed to be  $\mu$ which
will referred to as the {\em  uniform measure} in the following.  From now on and for  the
rest  of  the  paper,   $(I_n)$  and  $(X_n)$  are  assumed  to  be  independent 
i.i.d.\ sequences of random variables.

\begin{lemma}
The minimal solution $N$ of Equation~\eqref{statevol} is a stationary point process on
$H$: its distribution is invariant with respect to group operations, i.e.\ 
\[
\int_H f(xy) \,N(dy) \stackrel{\text{dist.}}{=} \int_H f(y) \,N(dy),
\]
for any $x\in H$ and any non-negative Borel function $f$ on $H$. 
\end{lemma}
\begin{proof}
By invariance of $\mu$ by translation,  property~{\em ii}) , the random variables $X_1$ and $xX_1$ have the same
distribution. If one denotes $\tau_x M$ the point process $M$ shifted by $x$,
i.e.\ $\tau_x M(\{y\})=M(\{x^{-1}y\})$ for all $y\in H$, then property~{\em i})  implies that the
sequence $(\tau_x \overline{N}_n)$ satisfies Relation~\eqref{rec1} with the variable $X_1$
replaced by $xX_1$. One concludes that the two sequences $(\tau_x \overline{N}_n)$ and
$(\overline{N}_n)$ have the same distribution and therefore that the same property holds
for their limits. The lemma is proved. 
\end{proof}
\begin{theorem}\label{Uniq}
For the LG and LO policies, if $p_+<1/2$ and the random variables $(X_i)$ are i.i.d.\ with
distribution $\mu$  on $H$, then, almost surely,  there exists a unique  point process $N$  satisfying Relation~\eqref{statevol} with finite
mass, i.e.\  $\P(N\in{\cal M}(H))=1$.   
\end{theorem}
\begin{proof}
By recurrence relation~\eqref{rec1}, one gets that for any Borel subset of $H$, 
\begin{multline*}
\E\left[\rule{0mm}{4mm}\overline{N}_n(A)\right]-\E\left[\rule{0mm}{4mm}\overline{N}_{n-1}(A)\right]
=p_+\,\P\left[\rule{0mm}{4mm}X_1\in A\right]\\
-(1-p_+)\int \P\left[\rule{0mm}{4mm}\overline{N}_{n-1}(B(x,1))\not=0, t_1(x,\overline{N}_{n-1})\in A\right]\,\mu(dx),
\end{multline*}
where, $\mu$ is the distribution of $X_1$. This identity with
$A=H$ and the monotonicity property give in particular that 
\[
(1-p_+)\int \P\left(N(B(x,1))\not=0\right)\,\mu(dx)\leq p_+.
\]
Since the distribution of $N(B(x,1))$ is, by the above lemma, independent of $x\in H$, one has 
\[
\P\left(N(B(x,1))\not=0\right)\leq \frac{p_+}{1-p_+}<1.
\]
The random variable $N(B(x,1))$ has therefore a positive probability of being $0$ and, in
particular, of being finite. The ergodicity property implies that for all $x\in H$,
$N(B(x,1))<+\infty$ almost surely. The point process $N$ has almost surely a finite mass,
$\P(N\in {\cal M}^*(H))=1$ since there is no accumulation point.  

Let $M$ be a point process satisfying Relation~\eqref{statevol} and $\P(M\in{\cal
  M}(H))=1$.  Equation~\eqref{statevol} gives that 
\[
M(H)\circ\theta-M(H)= \ind{I_1=+}-\ind{I_1=-,N(B(X_1,1))\not=0},
\]
since the right hand side is clearly integrable, the expected value of the left hand is
$0$. See Lemma~12.2 of Robert~\cite{Robert:08} for example. One gets the relation
\[
p_+=(1-p_+)\P(M(B(X_1,1))\not=0),
\]
one obtains the relation
\begin{equation}\label{zero}
\P(M(B(X_1,1))=0)=\frac{1-2p_+}{1-p_+}>0.
\end{equation}
The minimality property of $N$, cf. Theorem~\ref{ex}, gives that almost surely $N\ll M$ so that
\[
\{M(B(X_1,1))=0\}\subset\{N(B(X_1,1))=0\}.
\]
These two subsets having the same  probability by Equation~\eqref{zero}. 
Relation~\eqref{statevol} gives therefore that, almost surely,
\begin{align*}
(M(H){-}N(H))\circ\theta&=M(H){-}N(H)+\ind{I_1={-},N(B(X_1,1))\not=0}{-}\ind{I_1={-},M(B(X_1,1))\not=0}\\
&= M(H){-}N(H), \quad\text{a.s.}
\end{align*}
The non-negative random variable $M(H)-N(H)$ is invariant by $\theta$ and therefore is almost surely a
constant $C$ by the ergodicity property. Since $M(H)<+\infty$ almost surely, there exist
some $m\geq 1$ such that $\P(M(H)=m)>0$ and  some finite subset $\{x_1,\ldots,x_n\}$ of $H$ such
that
\[
H=\bigcup_{\ell=1}^n B(x_\ell,1),
\]
On the event $\{M(H)=m\}$, it is easily checked that if, for $\ell=1,\ldots,n$, a total of
$2m$
``$-$'' points are sent in each ball $B(x_\ell,1)$  and not $+$ occurs, then all the $m$
initial points will removed. More precisely,
\[
\{M(H)=m\} \bigcap_{\ell=1}^n\bigcap_{k=2m(\ell-1)}^{2m\ell-1}\{I_1\circ\theta^k=-, X_1\circ\theta^k\in B(x_\ell,1)\}
\subset\{M(H)\circ\theta^{2mn}=0\}.
\]
Since the variable is ${\cal F}_0$-measurable, it is independent of the sequence of
i.i.d.\  sequence $((I_1,X_1)\circ\theta^i,i\geq 0)$, one gets therefore that 
\[
0<\P(M(H)\circ\theta^{2mn}=0)=\P(M(H)=0)= \P(M(H)=0,N(H)=0),
\]
one deduces that the constant $C=M(H)-N(H)$ is $0$. The two point processes $M$ and $N$
coincide. The theorem is proved. 
\end{proof}

\begin{prop}[Convergence of Distributions of Configurations]\label{propconv}
For the LG, LR, LO policies, if $p_+<1/2$ and $P$ is some finite point process on $H$ and
$(M_n)$ is the sequence of point processes defined by, $M_0=P$ and 
\[
M_{n}=M_{n-1}+\ind{I_{n}=+}\delta_{X_{n}}-\ind{I_n=-,
  M_{n-1}(B(X_{n},1))\not=0}\delta_{t_1(X_{n},M_{n-1}))},\;n\geq 1,
\]
then $(M_n)$ converges in distribution to $N$, the unique solution of Equation~\eqref{statevol}.
\end{prop}
\begin{proof}
Recall that the sequence $(N_n)$  defined by Equation~\eqref{evol} corresponds to the case
where  the initial  state  is empty.   Let  $(\overline{M}_n)$ be  the  sequence of  point
processes satisfying Relation~\eqref{rec1} with $\overline{M}_0=P$. The sequence  $(\overline{N}_n)$
 defined by Relation~\eqref{rec1} is such that  $\overline{N}_0$ is the empty
state.     By   induction,    it   is    easy   to    check   that,    for    $n\geq   0$,
$\overline{M}_n=M_n\circ\theta^{-n}$ and $\overline{N}_n=N_n\circ\theta^{-n}$),
and therefore that $\overline{M}_n$ has the same distribution as $M_n$.

The monotonicity property gives  that  $\overline{N}_n\ll
\overline{M}_n$  holds and that if $m=P(H)$ is the 
number of initial points of $P$ then  necessarily
\[
0\leq \overline{M}_n(H)-\overline{N}_n(H)\leq m.
\]
The limit $N$ of $(\overline{N}_n)$ having a positive probability of being $0$,
there is almost surely an infinite number of $\ell\geq 0$  such that $N(H)\circ\theta^\ell=0$. 
The relation $\overline{N}_n(H)\leq N(H)$ implies  that there an infinite number of
$\ell\geq 0$ such that 
\[
N_\ell(H)=\overline{N}_\ell(H)\circ\theta^\ell=0.
\]
For these indexes $\ell$, $M_\ell(H)\leq m$ and, as in the proof of the above theorem, there is a
positive probability (lower bounded  by a quantity independent of the location of the
points of $M_\ell$) that all the $M_\ell(H)$ points are removed before a new ``$+$'' 
arrives. Hence, with probability $1$, there exists some (random) index $\ell$ such that
$M_n=N_n$. The convergence in distribution of $(M_n)$ is therefore proved. 
\end{proof}
\begin{corollary} 
The distribution of $N$ the solution of Equation~\eqref{statevol} is the only distribution
on ${\cal M}(H)$ invariant by the equation
\[
M_1=M_0+\ind{I_1=+}\delta_{X_1}-\ind{I_1=-}\delta_{t_1(X_1,M_0)}.
\]
\end{corollary}

\subsection*{Uniqueness Result when $H=\T_1(T)$}\ \\
This section is concluded with a uniqueness result for the one-dimensional torus. One
denotes by ${\cal M}^*_\mu(\T_1(T))$ the subset of elements $P$ of ${\cal
M}^*(\T_1(T))$  with a set $S(P)$ of accumulation points negligible for $\mu$, the
Lebesgue measure. The following proposition generalizes the uniqueness result of
Theorem~\ref{Uniq} for the solution of Equation~\eqref{statevol}. 
\begin{prop}[A Uniqueness Property for the Torus]\label{uniqtorus}
For the LG policy, if $M$  is  a random  variable in ${\cal M}^*_\mu(\T_1(T))$,
solution of Equation~\eqref{statevol},  ${\cal F}_0$-measurable  and such that, 
for any accumulation point $a\in S(M)$,
\[
M((a,a+\eps]) =+\infty \text{ and }M([a-\eps,a)) =+\infty
\]
holds almost surely for any $\eps>0$ sufficiently small, then $M$ is the minimal solution
$N$ of Equation~\eqref{statevol}. In particular the set $S(M)$ of accumulation points  is empty
\end{prop}
\begin{proof}
Assume that such a variable $M$ exists. It is known that $S(M)$ is a deterministic set and
$H-S(M)$ being an open set, it can be written as
\[
H-S(M)=\bigcup_{n\geq 1} (a_n,b_n),
\]
where $(a_n)$ and $(b_n)$ are sequences of elements of $S(M)$.
Note that, because of the assumption on $M$ near
accumulation points, the variable $t_1(x,M)$ is well defined (i.e.\ not equal to
$\partial$) for all $x\in (a_n,b_n)$, $n\geq 1$, as long as $M([x-1,x+1])\not=0$.  
The minimality of $N$ implies that $N\ll M$. From Equation~\eqref{statevol}, for $n\geq 1$
and $\eps$ sufficiently small,
\begin{multline*}
M([a_n+\eps,b_n-\eps])\circ\theta -M([a_n+\eps,b_n-\eps])\\=
\ind{I_{1}=+, X_{1}\in [a_n+\eps,b_n-\eps]}-\ind{I_1=-, M(B(X_{1},1))\not=0, t_1(X_{1},M)\in[a_n+\eps,b_n-\eps]}.
\end{multline*}
With the same argument as in the previous proof, one gets that the expected value of the
left hand side of the above identity is $0$ and consequently that,
\begin{align*}
\frac{p_+}{1-p_+}&\P(X_{1}\in [a_n+\eps,b_n-\eps]) \\
& = \P(M([X_{1}-1,X_1+1])\not=0, t_1(X_{1},N)\in[a_n+\eps,b_n-\eps])\\
& =  \P(M([X_{1}-1,X_1+1])\not=0, X_1\in(a_n,b_n), t_1(X_{1},M)\in[a_n+\eps,b_n-\eps]),
\end{align*}
due to the assumption on accumulation points.
By letting $\eps$ go to $0$ one gets the relation
\[
\frac{p_+}{1-p_+}\P(X_{1}\in [a_n,b_n))= \P(M([X_{1}-1,X_1+1])\not=0, X_1\in(a_n,b_n)),
\]
by summing up these terms with respect to $n$ and taking into account the fact that
$\mu(S(M))=0$, one finally obtains the identity 
\[
\P(M([X_{1}-1,X_1+1]))=0)=\frac{1-2p_+}{1-p_+},
\]
the same equality also holds for $N$, but since $N\ll M$, this implies that, for almost
every $x\in[0,T]$, 
\begin{multline*}
\P(M([x-1,x+1])=0)\\=\P(N([x-1,x+1])=0)=\P(N([-1,1])=0)=\frac{1-2p_+}{1-p_+}>0.
\end{multline*}
This is in contradiction with the fact that $M$ has accumulation points at some fixed
points. The proposition is proved.
\end{proof}

\section{The case of  the  Interval $[0,T]$}\label{sectorus}
In this  section one considers  a simple space,  the interval $[0,T]$, for  which boundary
effects occur contrary  to Section~\ref{HG} where the homogeneity  property rules out this
feature.  The  LG is analyzed in  this case  and stability  results are
proved. It should be  noted that the LO policy (Local  One-Sided), has a
completely different  qualitative behavior, it  is analyzed in Section~\ref{secmisc}  in a
more general setting. For the LR policy, Theorem~\ref{theoloc} addresses this case. 

The value  of $T$ is assumed  to be greater than  $1$, otherwise the  stability problem is
trivial.  The  location of  the points is  an i.i.d.\  sequence $(X_i)$ of  uniform random
variables   on    $[0,T]$.    The    variable   $N$   is    the   minimal    solution   of
Equation~\eqref{statevol}. As before, the ergodicity  property and Theorem~\ref{ex} give
that the set $S(N)$ of accumulation points of $N$ is deterministic.

Properties of possible accumulation points of $N$ are now analyzed in four steps. The set
$S(N)$ is assumed to be non-empty. 

\smallskip
\noindent (a)  {\em Accumulation points are at distance at least $1$.} \\ 
Assume that there  is at least two elements $a<b$ in $S(N)$ such
that  $b-a<1$ and $(a,b)\subset H{-}S(N)$.  Take  some  $\eps$ sufficiently small,  Equation~\eqref{rec1}  for  the
sequence $(\overline{N}_{n})$ gives the relation
\begin{multline*}
\E\left(\overline{N}_{n+1}([a-\eps,b+\eps])\circ\theta\right) 
-\E\left(\overline{N}_{n}([a-\eps,b+\eps])\right)\\=
p_+(b-a+2\eps)- \P(\overline{N}_{n}(B(X_1,1))\not=0,  t(X_1,\overline{N}_{n})\in[a-\eps,b+\eps]),
\end{multline*}
and by the monotonicity property of $(\overline{N}_{n})$,
\begin{multline*}
\P\left(\overline{N}_{n}(B(X_1,1))\not=0, X_1\in[a+\eps,b-\eps],
t(X_1,\overline{N}_{n})\in[a-\eps,b+\eps]\right)\\ \leq \frac{p_+}{1-p_+} (b-a+2\eps).
\end{multline*}
Since $a$ and $b$ are almost surely accumulation points, the non-decreasing sequence of sets
\[
\left\{\overline{N}_{n}(B(X_1,1))\not=0, X_1\in[a+\eps,b-\eps],  t(X_1,\overline{N}_{n})\in[a-\eps,b+\eps]\right\}
\]
converges, as $n$ goes to infinity, to the set $\{a+\eps\leq X_1\leq b-\eps\}$. By taking
the limit in the last inequality, 
one gets the relation 
\[
b-a-2\eps\leq \frac{p_+}{1-p_+}(b-a+2\eps),
\]
by letting $\eps$ go to $0$, this gives $p_+\geq 1/2$. Contradiction. 
Consequently, if there are  accumulation points for $N$, they are  isolated points of
$[0,T]$ at distance $1$ at least.

\smallskip
\noindent (b) {\em Coupling.} \\
One denotes temporarily by $N_T$ the point process, solution of Equation~\eqref{statevol}, associated to uniform random variables
on $[0,T]$.  Let $S<T$ and $(\overline{N}_{n}^S)$ be the sequence defined by 
\[
\overline{N}_{n+1}^S\circ \theta=\overline{N}_{n}^S+ \ind{I_{1}=+1, X_1\leq S}\delta_{X_{1}}-\ind{I_1=-1,    
\overline{N}_{n}(B(X_{1},1))\not=0, X_1\leq S}\delta_{t_1(X_{1},\overline{N}_{n}^S)}.
\]
then, almost surely, 
\[
\lim_{n\to+\infty} \overline{N}_{n}^S\stackrel{\text{def.}}{=}M_S\stackrel{\text{dist.}}{=}N_S.
\]
By induction, it is easily checked that $\overline{N}_n([0,S]\cap \cdot)\ll 
\overline{N}_n^S$ holds for all $n\geq 1$. The inequality $\ll$ instead of equality comes
from the fact that for $\overline{N}_n([0,S]\cap \cdot)$ a minus arriving in $[S,T]$ can
kill a point in $[0,S]$. By letting $n$ go to infinity, one obtains the relation
\begin{equation}\label{Joblot07}
{N}([0,S]\cap \cdot)\ll M_S.
\end{equation}

\smallskip
\noindent (c)  {\em Patterns of Accumulation Points.} \\
Denote by $S_+(N)$ the subset of elements of $S(N)$ which have an infinite number of
points of $N$ on their right,
\[
S_+(N)=\{x\in [0,T]: \exists \eps_0>0, \forall \eps\leq \eps_0, N((x,x+\eps))=+\infty \text{
  a.s. }\}.
\]
Similarly $S_-(N)$ is defined for left neighborhoods.

\smallskip
\noindent
{\em Claim:} There do not exist $a\in S_+(N)$ and $b\in S_-(N)$ such that $(a,b)\subset [0,T]-S(N)$.
Assume there are such $a$ and $b$.
Equation~\eqref{statevol} for $N$ gives the relation 
\begin{multline*}
N\circ \theta-N\\=
\ind{I_{1}=+,a<X_1<b}\delta_{X_{1}}-\ind{I_1=-,N([X_1-1,X_1+1])\not=0,a<t_1(X_{1},N)<b }\delta_{t_1(X_{1},N)},
\end{multline*}
valid in the space ${\cal M}([a,b])$. Because of the assumption on $a$ and $b$ and of the
definition of $t_1(\cdot,N)$, one has the 
identity
\[
\{N([X_1-1,X_1+1])\not=0,a{<}t_1(X_{1},N){<}b \}{=}\{N([X_1-1,X_1+1])\not=0,a{<}X_{1}{<}b \}
\]
which gives the relation
\begin{multline}\label{eqaux}
N\circ \theta-N\\=
\ind{I_{1}=+,a<X_1<b}\delta_{X_{1}}-\ind{I_1=-,N([X_1-1,X_1+1])\not=0,a<X_{1}<b }\delta_{t_1(X_{1},N)}.
\end{multline}
If $a$ and $b$ are identified, the above equality states that the point process $N$
restricted to the torus $\T_1(b-a)=[a,b]$ satisfies Relation~\eqref{statevol} for this
state space. Proposition~\ref{uniqtorus} shows that $N$ on $[a,b]$ is a point process with
finite mass. Contradiction.

\smallskip
\noindent (d)  {\em Conclusion.} \\ 
Recall that $N$ is the limit of the $({\overline N}_n)$ when the sequence $(X_i)$ is
 i.i.d.\ uniformly distributed on $[0,T]$. Since $(X_i)$ has the same distribution as
 $(T-X_i)$,  one deduces that the distribution of $N$ is invariant with respect to the
 mapping $x\mapsto T-x$.  Accumulation points of $N$ being at distance $1$ at least by
 (a), $S(N)$ is a finite set, $S(N)=\{a_1,\ldots, a_p\}$, for some $p_+\geq 1$ and $0\leq
 a_1<a_2<\cdots<a_p\leq T$.  

Assume that  $a_1\in S_+(N)$ holds. 
By symmetry of $N$ with respect to the mapping $x\to {T}-x$, one gets that $a_p=T-a_1\in S_-(N)$.
By (c) one has necessarily $a_2\in S_+(N)$, and therefore $a_{p-1}\in S_-(N)$. By
proceeding inductively, one deduces that there exists a $k<p$ such that $a_k\in S_+(N)$
and $a_{k+1}\in S_-(N)$. This is impossible according to (c). 

Consequently, $a_1\in S_-(N)$ and $a_1>0$.
By the coupling result~\eqref{Joblot07} above, with the same notations as in (b), one has 
\[
N_T([0,a_1]\cap \cdot)\ll N_{a_1}.
\]
In particular $a_1$ is an accumulation point of $N_{a_1}$, by symmetry of $N_{a_1}$ with respect to
the mapping $x\to {a_1}-x$, one gets that $0$ is also an accumulation point of
$N_{a_1}$. Consequently, $S(N_{a_1})=\{b_1,\ldots,b_q\}$, with $b_1\in S_+(N_{a_1})$ which
is impossible by what have just been proved. The set $(N)$ is therefore
empty.

The uniqueness statement of the following proposition has therefore been proved. 
\begin{prop}[Stability property for  the LG policy]\label{94}\ \\
If $p_+<1/2$ and the random variable $X_1$ is uniformly distributed on $[0,T]$, 
then Equation~\eqref{statevol}  has a unique  minimal solution $N$ such that $\P(N\in{\cal
  M}(H))=1$. 

If $N_0$ is an element of ${\cal M}([0,T])$ with finite mass, then the sequence $(N_n)$
defined by Recursion~\eqref{evol} converges in distribution to $N$. 
\end{prop}
\begin{proof}
The proof of the convergence in distribution follows the same lines as in the proof of
Proposition~\ref{propconv}.
\end{proof}

\noindent
The distribution of the variable $X_1$ is now assumed to have a
density $h$ with respect to Lebesgue's Measure. 
\begin{prop}[Non-Uniform Distributions]
If $p_+<1/2$ and the distribution of $X_1$ has density $h$ on $[0,1]$ which is  piecewise
constant on a finite partition of sub-intervals of $[0,T]$, then the conclusions of
Proposition~\ref{94} also hold in this case.  
\end{prop}
\begin{proof}
The proof is sketched since most of the arguments have been already used at several occasions.
By assumption there is a partition of $[0,T]$ by sub-intervals $(I_k,k\in K)$ and
$(\alpha_k,k\in K)$ such that, for $x\in[0,T]$,
\[
h(x)=\sum_{k\in K} \alpha_i {\mathbbm 1}_{I_k}(x).
\]
The sequence $(\overline{N}_n)$ defined by Recurrence~\eqref{rec1} can be dominated by the
sequence of point process $(\widetilde{N}_n)$ whose dynamic is modified as follows: a minus point
falling into some sub-interval $I_k$ n does not kill a point in another sub-interval.  In this
way, for $n\geq 1$, one has clearly $\overline{N}_n\ll \widetilde{N}_n$. Now, for $k\in
K$, the point process $\widetilde{N}_n$ restricted to $I_k$ is, up to a translation, simply  the point process
associated to uniformly distributed random variables on $I_k$ when 
\[
\ind{X_1\in I_k}+\ind{X_2\in I_k}+\cdots+\ind{X_n\in I_k}
\]
points have been used. By Proposition~\ref{94}, the point processes $\widetilde{N}_n$,
$n\geq 1$ are upper bounded by a point process with finite mass. This shows that 
$N$, the limit of $(\overline{N}_n)$, has almost surely a finite mass. 
\end{proof}

\section{One-Sided Policies}\label{secmisc}
In this section, a multi-dimensional generalization of the results of
Robert~\cite{Robert:03} is presented. It is assumed that 
$T=(T_i)\in\R_+^d$ with $T_i>0$ for $1\leq i\leq d$,  $H$ is  defined as
\[
H=\prod_{i=1}^d [0,T_i],
\]
and that the locations of the points $(X_i)$ are uniformly distributed in $H$. 

With a slight abuse  of notation, one will denote $H=[0,T]^d$ and if $x$, $y\in\R_+$, $xy$
[resp. $x/y$] will stand for $(x_iy_i)$ [resp. $(x_i/y_i)$]. Similarly, if $x=(x_i)\in\R_+^d$,
$\log x$ denotes $(\log x_i)$ and  finally $\Delta$ is the
subset defined as the lower boundary of $H$, 
\[
\Delta\stackrel{\text{def.}}{=}\{x\in[0,T]^d:\exists i\in\{1,\ldots,d\}, x_i=0\}.
\]
A ``$-$'' particle at $x$ can only kill the closest particle of the point process in 
the orthant with the corner at $x$, i.e.\ in the set $(x+\R_+^d)\cap H$. 
In order to get a more precise characterization of the variable $N$ of Theorem~\ref{ex}, the
following notation has to be introduced. If $M\in{\cal M}^*(H)$, one denotes by $D(M)$ the
``dead zone'' of $M$ for minus particles, i.e. the set of locations where no point of $M$ can be
killed by them,
\[
D(M)=\{y\in H: M\left((y+\R_+^d)\cap H\right)=0\}.
\]
If $M$ is the null measure, then $D(M)=H$ and if $P$, $Q\in{\cal M}^*(H)$ are such that
$P\ll Q$, then $D(Q)\subset D(P)$. 

In this context, the corresponding stationary evolution equation is given by
\begin{equation}\label{statevolO}
N\circ\theta=N+ \ind{I_{1}=+}\delta_{X_{1}}-\ind{I_1=-,X_1\not\in D(N) }\delta_{t_1(X_{1},N)}.
\end{equation}
With the same arguments as in  Theorem~\ref{ex} for local policies, there exists a unique minimal $N$ in the
set ${\cal M}^*(H)$ with probability $1$ which is solution of
Equation~\eqref{statevolO}. The variable $N$ is the limit of the non-decreasing sequence
$(\overline{N}_n)$ defined by the recurrence 
\begin{equation}\label{rec2}
\overline{N}_{n+1}\circ \theta=\overline{N}_{n}+ \ind{I_{1}=+1}\delta_{X_{1}}-\ind{I_1=-1,    
X_1\not\in D(\overline{N}_{n})}\delta_{t_1(X_{1},\overline{N}_{n})}.
\end{equation}
The following proposition establishes a specific property of this policy, namely that it
exhibits an invariance with respect to scaling. 
\begin{prop} If $p_+<1/2$ and $\alpha=(\alpha_i)\in\R_+^d$ with $0<\alpha_i\leq 1$ for $1\leq
i\leq d$, for the GO policy the minimal solution $N$ of
Equation~\eqref{statevolO} satisfies the invariance relation
\begin{equation}\label{scale}
\int_{[0,\alpha T]^d}f(x)\,N(dx)\stackrel{\text{dist.}}{=} \int_{[0,T]^d} f(\alpha x)\,N(dx)
\end{equation}
for any continuous function on $[0,T]^d$. Furthermore $N$ is almost surely a Radon measure on
$[0,T]^d-\Delta$.% and it is the unique solution of Equation~\eqref{statevolO} with this property.
\end{prop}
\begin{proof}
Relation~\eqref{scale} is a consequence of the two following simple facts:
\begin{itemize}
\item  The variables $(X_i)$ that are in $[0,\alpha T]$ have the same distribution as
  $(\alpha X_i)$.
\item  Invariance by scaling of the dynamics:
\[
\int_{[0,\alpha T]^d} f(x)\overline{N}_{\nu_n}(dx) \stackrel{\text{dist.}}{=}  
\int_{[0,T]^d}f(\alpha x)\overline{N}_{n}(dx),
\]
where $\nu_n$ is the first index $k$ for which exactly $n$ elements of the $k$ first
points are in $[0,\alpha   T]^d$. 
Relation~\eqref{scale} follows from this identity by letting $n$ go to infinity.
\end{itemize}

Equation~\eqref{rec2} gives the inequality
\[
0\leq \E\left(\overline{N}_n(H)\right)-\E\left(\overline{N}_{n-1}(H)\right)
=p-(1-p)\P(X_1\not\in D(\overline{N}_{n-1})).
\]
By letting $n$ go to infinity and by using the fact that the non-increasing sequence of
sets $(D(\overline{N}_{n-1}))$ is converging to $D(N)$, one gets therefore that
\[
\P(X_1\not\in D(N))\leq \frac{p_+}{1-p_+}<1.
\]
The set $D(N)$ is therefore non-empty with some positive probability.

With the ergodicity property, any accumulation point $a=(a_i)\in[0,T]^d$ of $N$ is
deterministic.  Assume that $a_i>0$ for all $1\leq i\leq d$. 
One considers the case where all the $a_i$ are such that $a_i<T_i$, the others situations
are treated in a similar way by using one-sided neighborhoods of $a$.
Take  $\eps_0>0$ sufficiently  small so that, if $\eps<\eps_0$, then
$a+\eps\stackrel{\text{def.}}{=}(a_i+\eps)$ and $a-\eps\in[0,T]^d$.  One denotes by
$(\alpha_i)=(T/(a_i+\eps))$, then with probability $1$,  $N([a-\eps,a+\eps]^d)=+\infty$
for all $0<\eps<\eps_0$. Relation~\eqref{scale} implies therefore that $T$ is also an
accumulation point. This contradicts that the fact that the set $D(N)$ is therefore
non-empty with some positive probability. One concludes that if there exists an
accumulation point of $N$, then  necessarily one of its coordinates is null and therefore
it belongs to $\Delta$. This shows $\P(N\in{\cal M}([0,T]^d-\Delta))=1$.

\begin{comment}
If $M$ is a solution of Equation~\eqref{statevolO} such that $\P(M\in{\cal
  M}([0,T]^d-\Delta))=1$, with similar arguments as before, one gets
\[
\P(X_1\not\in D(M))= \P(X_1\not\in D(N))={p_+}/{(1-p_+)}.
\]
By minimality of $N$, $N\ll M$ and therefore  $D(M)\subset D(N)$, the above identity shows
that $D(M)=D(N)$ holds almost surely. By definition of $D(M)$ [resp. $D(N)$],  there exist a finite
number of points  of $M$ [resp. $N$], $(u^1,\ldots,u^p)$ [resp. $(v^1,\ldots,v^q)$] such that
\[
H-D(N)=\bigcap_{i=1}^q [0,v^i]^T, \text{ and } H-D(M)=\bigcap_{i=1}^p [0,u^i]^T
\]
and the $u^i$'s [resp. $v^i$'s] are on the boundary of $H-D(M)$ [resp. $H-D(N)$]. The
identity $D(M)=D(N)$ implies that 
\end{comment}
\end{proof}
The following proposition shows that for the invariant distribution,
configurations under the GO policy have an infinite number of points with probability
$1$. It will be shown that this property also holds for the local version of the policy.
\begin{prop}[Infinite number of points near $\Delta$]
Almost surely, any point of the set $\Delta$ is an accumulation point of the solution $N$
of Equation~\eqref{statevolO} for the GO policy. Furthermore, the point process
$\widetilde{N}$ on $\R_+^d$ defined by 
\[
\widetilde{N}=\int_{[0,T]^d} \delta_{-\log(u/T)}\,N(du).
\]
is a stationary point process on $\R^d_+$, i.e.\ for $x\in\R_+^d$, the distribution of the
variable  $\widetilde{N}$ is invariant with respect to the translation to $x$:
\[
\int_{\R^d} f(x+y)\,\widetilde{N}(dy)\stackrel{\text{dist.}}{=}\int_{\R^d} f(y)\,\widetilde{N}(dy),
\]
for any continuous function $f$ with compact support on $\R_+^d$.
\end{prop}
\begin{proof}
Let $a\in\Delta$, it is assumed, that for example, only the first coordinate is $0$,
$a=(0,a_2,\ldots,a_d)$. Let $0<\delta\leq 1$, $\eps>0$ and denote by
\[
A_\delta=[0,\delta T_1]\times \prod_{i=2}^d [a_i,a_i+\eps],
\]
by taking $\alpha=(\delta,1,\ldots,1)$ and using Relation~\eqref{scale}, one gets the
identity
\[
N\left(A_\delta\right)\stackrel{\text{dist.}}{=} N\left(A_1\right)
\]
for all $0<\delta\leq 1$. Since $\P(N(A_\delta)<+\infty)$ is either $0$ or $1$ and since 
clearly $\P(N(A_1-A_\delta)\not=0)>0$, one gets that $\P(N(A_\delta)=+\infty)=1$. By
the above proposition, one has that $N(A_1-A_\delta)$ is almost surely finite. One
concludes that $a$ is an accumulation point. Consequently, the same property holds when
there are several coordinates which are $0$.

Relation~\eqref{scale} gives that, for $\alpha\in[0,1]^d$, the identity
\begin{multline*}
\left(N\left(\prod_{i=1}^d [\alpha_i y_i,\alpha_i x_i]\right),  x, y\in[0,T]^d, y\leq x\right)\\
\stackrel{\text{dist.}}{=}
\left(N\left(\prod_{i=1}^d [y_i, x_i]\right),  x, y\in[0,T]^d, y\leq x\right)
\end{multline*}
holds. By taking $z=-\log \alpha$, this relation can be rewritten as
\begin{multline*}
\left(\widetilde{N}\left(\prod_{i=1}^d [ v_i+z_i, u_i+z_i]\right),  u, v\in\R_+^d, v\leq u\right)\\
\stackrel{\text{dist.}}{=}
\left(\widetilde{N}\left(\prod_{i=1}^d [v_i, u_i]\right),  u, v\in[0,T]^d, v\leq u\right).
\end{multline*}
The point process $\widetilde{N}$ is invariant with respect to the non-negative translations.

\end{proof}

\begin{corollary}[Local One-Sided Policy on the torus $\T_1(T)$]\ \\
The minimal solution $N_{L}$ of the equation
\[
N_{L}\circ\theta=N_{L}+ \ind{I_{1}=+}\delta_{X_{1}}-\ind{I_1=-,X_1\not\in D(N_{L}), t_1(X_1,N_{L})\in B(X_1,1) }\delta_{t_1(X_{1},N_{L})}.
\]
for the LO policy is such that  $\P(N_L\in {\cal M}((0,T]))=1$ and every element of $\Delta$ is almost surely an accumulation point of $N$. 
\end{corollary}
\begin{proof}
With the same arguments as before, it is not difficult to prove that the solution $N$ of
Equation~\eqref{statevolO}, is such that $N\ll N_{L}$ which gives the result for the
accumulation points. The proof that $N_L$ is a Radon measure on $(0,T)$ with probability
$1$ is sketched. As before, one first proves that accumulation points are at distance $1$
at least. If there is another accumulation point than $0$, denote by $a>0$ the smallest
which is not $0$, by considering the evolution of the number of points on the interval $[a-1,a+\eps]$
for some $\eps>0$, it is not difficult to get a contradiction to the fact that $p_+<1/2$. 
\end{proof}

\providecommand{\bysame}{\leavevmode\hbox to3em{\hrulefill}\thinspace}
\providecommand{\MR}{\relax\ifhmode\unskip\space\fi MR }
% \MRhref is called by the amsart/book/proc definition of \MR.
\providecommand{\MRhref}[2]{%
  \href{http://www.ams.org/mathscinet-getitem?mr=#1}{#2}
}
\providecommand{\href}[2]{#2}

\end{document}